\documentclass[reqno,12pt]{amsart}
\usepackage{latexsym,amsmath,amssymb,amsbsy,amscd,bm,amsfonts}
\usepackage{hyperref}
\hypersetup{
  colorlinks   = true, 
  urlcolor     = blue, 
  linkcolor    = blue, 
  citecolor    = red   
}
\theoremstyle{plain}
\newtheorem{theorem}{Theorem}[section]

\newtheorem{lemma}{Lemma}[section]
\newtheorem{prop}{Proposition}[section]
\newtheorem{cor}{Corollary}[section]

\theoremstyle{definition}
\newtheorem{definition}{Definition}[section]

\theoremstyle{remark}

\newtheorem{remark}{Remark}[section]

\numberwithin{equation}{section}

\newcommand{\real}{{\mathbb R}}
\newcommand{\roots}{{\mathcal R}}
\newcommand{\proots}{{\mathcal R}^+}
\newcommand{\comp}{{\mathbb C}}

\newcommand{\G}{{\mathbb G}}
\newcommand{\Lie}{{\mathcal G}}
\newcommand{\T}{{\mathcal T}}
\newcommand{\To}{{\mathbb T}}
\newcommand{\Ho}{{\mathcal H}}

\newcommand{\La}{{\mathcal L}}

 \DeclareMathOperator{\Ad}{Ad}
 \DeclareMathOperator{\ad}{ad}

 \DeclareMathOperator{\Aut}{Aut}
 \DeclareMathOperator{\End}{End} 
 \DeclareMathOperator{\spa}{span}
 \DeclareMathOperator{\tr}{trace}

  \DeclareMathOperator{\SU}{SU}
   \DeclareMathOperator{\U}{U}
     \DeclareMathOperator{\su}{su}
  \DeclareMathOperator{\GL}{GL}
 \DeclareMathOperator{\gl}{gl}
 \DeclareMathOperator{\Le}{Length}
 
\begin{document}
\title[Length spectra of sub-Riemannian metrics on Lie groups]{Length spectra of sub-Riemannian metrics on compact Lie groups}
\author[Domokos, Krauel, Pigno, Shanbrom, VanValkenburgh]{Andr\'{a}s Domokos, Matthew Krauel, Vincent Pigno, \newline Corey Shanbrom, Michael VanValkenburgh}
\address{Department of Mathematics and Statistics,
California State University Sacramento, 6000 J Street, Sacramento, CA, 95819, USA}
\email{domokos@csus.edu, krauel@csus.edu, \newline vincent.pigno@csus.edu, corey.shanbrom@csus.edu, mjv@csus.edu}

\date{\today}

\keywords{sub-Riemannian geometry, geodesics, root systems, compact Lie groups}

\subjclass[2010]{53C17, 53C22, 22E30, 51N30}

\begin{abstract}

Length spectra for Riemannian metrics are well studied,  while sub-Riemannian length spectra have been largely unexplored.
Here we give the length spectrum for a canonical sub-Riemannian structure attached to any compact Lie group by restricting its Killing form to the sum of the root spaces. 
Surprisingly, the shortest loops are the same in both the Riemannian and sub-Riemannian cases.
We 
provide specific calculations for  $\SU (2)$ and $\SU (3)$.

\end{abstract}

\maketitle

\section{Introduction}\label{sec:1Intro}

\noindent
While much is known about the existence and geometric properties of closed geodesics on Riemannian manifolds in general \cite{kling}, and Lie groups in particular, we cannot say the same thing about their connection with the algebraic structure of Lie groups. Moreover, the sub-Riemannian setting has been mostly neglected.  

In the case of simple, simply connected, compact Lie groups, Helgason obtained the length of the shortest Riemannian geodesic loop in terms of the length of the highest root \cite[Proposition 11.9]{hel}.
We expand upon Helgason's work using more algebraic methods, obtaining the sub-Riemannian and Riemannian geodesic loop length spectra. The sub-Riemannian structure consists of the horizontal distribution defined by the orthogonal complement of a Cartan subalgebra and the restriction of the bi-invariant metric defined by the Killing form.  To our knowledge, nothing was previously known about the connection between root systems and lengths of sub-Riemannian geodesic loops.  

In Section \ref{sec:2Background} we provide the background for the root space decomposition of semi-simple, compact Lie algebras and prove Theorem \ref{thm:SubRiemannianGeodesicsNormal}, which shows that all sub-Riemannian geodesics are normal. In Section \ref{sec:3LengthsSR} we work in a simple, simply connected, compact Lie group. We find connections between the algebraic information encoded in the root system of the Lie algebra and properties of Riemannian and sub-Riemannian geodesic loops. In Theorems \ref{thm:RGeodesics/RPrime} and \ref{thm:PurelySubRLengths} we describe the entire length spectra of the Riemannian and certain sub-Riemannian geodesic loops.  
In Theorem \ref{thm:Xdecomposition} we find properties that help describe the remaining sub-Riemannian geodesic loops. 
Moreover, in Theorem \ref{thm:shortestlengths}, we compute the lengths of the shortest Riemmanian and sub-Riemannian loops, which unexpectedly turn out to be equal. Further, in Corollary \ref{cor:RootMax} we derive a purely algebraic formula for the length of the highest root.
In Sections \ref{sec:4SU(2)} and \ref{sec:5SU(3)} we provide relevant examples in $\SU (2)$ and $\SU (3)$.

Note that the terms \emph{length spectrum} and \emph{geodesic} have different variants in the literature.  By length spectrum, we mean the set of lengths of all primitive geodesic loops.  A sub-Riemannian geodesic is defined as in \cite{mont} as a locally length minimizing curve.  While in general such curves may not satisfy the geodesic equations, in our setting we show that the two notions coincide (see Theorem \ref{thm:SubRiemannianGeodesicsNormal}).  

\section{General results}\label{sec:2Background}
\noindent In this section we assume that $\G$ is a semi-simple, connected, compact matrix Lie group. This assumption is suited to present and prove some general results about sub-Riemannian geodesics, and we will use the more restrictive simple and simply connected assumptions in the following sections, where we prove results about sub-Riemannian geodesic loops.  
Our notations and definitions will be geared toward the presentation of the sub-Riemannian geometry, rather than the algebraic theory of Lie groups.

 The Lie algebra of $\G$ can be defined in terms of the matrix exponential:
$$\Lie = \left\{ X  \in {\mathcal M}_n  \;  :  \; e^{t X} \in \G \; , \; \; \forall \; t \in \real \right\} \, ,$$
where ${\mathcal M}_n$ is the linear space of $n \times n$ real or complex matrices in which $\G$ is included.
Then $\Lie$ is a real Lie algebra endowed with the commutator operator 
$$[X,Y] = XY-YX \, .$$

A Lie algebra is called simple if it is non-commutative and does not have any non-trivial ideals, and it is called semi-simple if it is the direct sum of simple Lie algebras. A Lie group is simple or semi-simple if its Lie algebra has the corresponding property. 

The adjoint representation of $\G$ is the group homomorphism
$$\Ad \colon \G \to \Aut (\Lie ) \, , \; \; \Ad (g) (X) = g X g^{-1} \, ,$$
while its differential at the identity is the adjoint representation of its Lie algebra
$$\ad \colon \Lie \to \End (\Lie) \, , \; \; \ad X (Y) = [X,Y] \, .$$ 
Note that, among semi-simple Lie algebras, $\Ad$ is an irreducible representation of $\G$ if and only if $\G$ is simple.

The Killing form  
$$K(X,Y) = \tr (\ad X \cdot \ad Y ) \, ,$$
is negative definite and non-degenerate on the Lie algebra of a semi-simple, compact Lie group, and hence we can define an inner product on $\Lie$ as
\begin{equation}\label{eq:InnerProduct}
\langle X , Y \rangle = - \rho \, K(X,Y) \, ,
\end{equation}
where $\rho> 0$ is a constant which can be adjusted according to our normalization preferences. The inner product \eqref{eq:InnerProduct} generates a bi-invariant metric on $\G$.
The Killing form is $\Ad$-invariant, so $\Ad (g)$ is a unitary linear transformation of $\Lie$ for all $g \in \mathbb G$ and $\ad X$ is skew-symmetric for all $X \in \Lie$.

Let $\To$ be a maximal torus in $\G$ and $\T$ be its Lie algebra. In this case, $\T$ is a maximal commutative subalgebra of $\Lie$ called the Cartan subalgebra. Its dimension is called the rank of $\Lie$, and also the rank of $\G$.
Consider an orthonormal basis ${\mathcal B}_{\T} = \{ T_1, \dots , T_r \} $ of $\T$, which will be fixed throughout the paper.
 
We extend the inner product \eqref{eq:InnerProduct} on  $\Lie$ bi-linearly to the complexified Lie algebra $\Lie_{\comp} = \Lie \oplus i \Lie$.  
The mappings $\ad T \colon \Lie_{\comp} \to \Lie_{\comp}$, $T \in \T$, commute and are skew-symmetric, so they share eigenspaces and  have purely imaginary eigenvalues.

Once we fix the orthonormal basis ${\mathcal B}_{\T}$, we can identify $\T^*$ with $\T$ and define the roots as elements of the Cartan subalgebra, as in \cite{dom15}. 

\begin{definition}\label{def:Root}
We define $R \in \T$ to be a root if $R \neq 0$ and the root space $\Lie_{R} \neq \{0\}$, where
\begin{equation}\notag  \label{eq:RootSpace}
\Lie_{R} = \{ Z \in \Lie_{\comp} \; : \; \ad T (Z) = i \; \langle R, T\rangle \, Z  \, , \; \; \forall \; \; T \in \T \;  \}\, .
\end{equation}
Additionally, we use the notation $\Lie_0 = \T_{\comp}=\T \oplus i\T$.
\end{definition}

Let $\roots$ be the set of all roots, which will be partially ordered by the relation $R_1 > R_2$ if  the first non-zero coordinate of $R_1 - R_2$ relative to the ordered basis ${\mathcal B}_{\T}$ is positive. We call a root positive if its first non-zero coordinate is positive and let $\proots$ denote the set of all positive roots.
For the most important properties of  $\Lie_{R}$ we quote \cite{dui,knapp}:
\begin{align*}
(i) &  \; \; \; \dim_{\comp} \Lie_{R} =1. \nonumber\\
(ii) & \; \; \; \text{If} \; R \in \roots, \; \text{then} \; -R \in \roots. \nonumber\\
(iii) & \; \; \; \Lie_{-R} = \{ X - iY \; : \; X + iY \in \Lie_{R} \}. \nonumber \\
(iv) & \; \; \; \langle \Lie_{R_1}, \Lie_{R_2} \rangle = 0  \; \; \text{if} \; \;  R_1, R_2 \in \roots \cup \{0\}  \, , \; R_1 \neq \pm R_2 .\\
(v) & \; \; \; [\Lie_{R_1} , \Lie_{R_2} ] \; \;  \left\{ \begin{array}{ll} = \Lie_{R_1 + R_2} \; &\text{if} \; \; R_1 + R_2 \in \roots\\ = \{0\} \;  &\text{if} \; \; R_1 + R_2 \not\in \roots \, \text{and} \; R_1 + R_2  \neq 0\\ \subset  i \T \;  &\text{if} \; \; R_1 + R_2 =0. \end{array} \right. \nonumber \\
(vi) & \; \; \; \text{If}  \; \; Z_{R} \in \Lie_{R}\; \; \text{and} \; \;  Z_{-R} \in \Lie_{-R} \, , \; \;  \text{then} \; \; [Z_{R}, Z_{-R}]= i \, \langle Z_{R} , Z_{-R} \rangle \; R .\nonumber 
\end{align*}
The above properties of $\Lie_{R}$ and the real root space decomposition
 $$\Lie = \T \oplus \Ho \, $$ where
 \begin{equation}\label{eq:LieAlgebraDecomposition}
\Ho = \T^{\perp} = \bigoplus_{R \in \proots} \, (\Lie_{R} \oplus \Lie_{-R} ) \cap \Lie \, ,
 \end{equation}
allow us to choose an orthonormal basis of $\Ho $,
\begin{equation}\label{eq:HorizontalBasis}
{\mathcal B}_{\Ho} = \{ X_1, \dots , X_k, Y_1, \dots , Y_k \} \, ,
\end{equation}
with the following properties:
\begin{equation} \label{eq:HorizontalBasisProperties}
 \begin{aligned}
 (i) & \; \; \text{For all} \; \; 1\leq j \leq k \; \; \text{there exists}  \; \; R_j \in \proots \; \; \text{such that} \\
 &\; \; \{X_j, Y_j \} \subset (\Lie_{R_j}\oplus \Lie_{-R_j}) \cap \Lie .\\
(ii) & \; \; E_{\pm j} = X_{j} \pm iY_j \in \Lie_{\pm R_j} .  \\
(iii) & \; \; \langle E_{j} , E_{-j} \rangle = 2 .\\
(iv) & \; \; [X_{j}, Y_{j}]= - R_j \, .
\end{aligned}
\end{equation}

Notice that  $\{ (g , \Ho_g ) \, : \, g \in \G\}$, where $\Ho_g = g \Ho$, forms a sub-bundle of the tangent bundle of $\G$, which we call the horizontal sub-bundle. The property $\T \subset [\Ho , \Ho ] $ shows that this horizontal sub-bundle is bracket-generating, hence its choice defines a sub-Riemannian metric on $\G$ in the following way (see \cite{mont}).

We call an absolutely continuous curve $\gamma \colon [a,b] \to \G$ horizontal if $\gamma ' (t) \in \Ho_{\gamma (t)}$ for every $t \in [a,b]$ where $\gamma ' (t)$ exists.
The length of a horizontal curve is defined as
$$\Le (\gamma) = \int_a^b \lVert \gamma' (t) \rVert \, dt \, .$$
The bracket-generating property implies that any two points can be connected by horizontal curves and therefore we can define a sub-Riemannian (also called Carnot-Carath\'{e}odory) distance as 
$$d(x,y) = \inf \{ \Le (\gamma ) : \, \gamma \text{ is a horizontal curve connecting $x$ and $y$} \} \, .$$   

We say that a horizontal curve $\gamma$ is a sub-Riemannian geodesic if locally it is a length minimizer. We call a sub-Riemannian geodesic $\gamma \colon [0,1] \to \G$ a sub-Riemannian geodesic loop if $\gamma (0) = \gamma (1) = I$ and $\gamma (t) \neq I$ for all $t \in (0,1)$. Here, $I$ denotes the identity matrix.

If we do not restrict the curve $\gamma$ to be horizontal, then similar definitions lead to Riemannian geodesics and geodesic loops.
With the choice of the bi-invariant inner product \eqref{eq:InnerProduct}, the Riemannian geodesics through the identity and in the direction of an arbitrary $X \in \Lie$ have the form (see \cite[Chapter 3]{arv})
$$\gamma (t) = e^{tX} \, .$$

\begin{remark}\label{remark:Smooth}
With our assumptions on $\G$ and $\Ho$, all sub-Riemannian geodesics are smooth \cite[Theorem 3]{mont1}. Moreover, as the inner product on $\Ho$ is the restriction of the inner product \eqref{eq:InnerProduct} defined on $\Lie$, a sub-Riemannian geodesic is also a smooth curve of equal Riemannian length.
\end{remark}

Sub-Riemannian geodesics  can be characterized in various ways. We follow the description from \cite{mont1,mont2,mont}, but also see \cite{agra1,bosc1}.
If a sub-Riemannian geodesic is a projection to $\G$ of a solution to Hamilton's equations for the sub-Riemannian Hamiltonian, then we call it normal,  otherwise we call it abnormal.  If a sub-Riemannian geodesic is a critical point of the endpoint map, then we call it singular, otherwise we call it regular \cite{mont1}. The following implications hold.
\begin{prop}\label{prop:AbnormalAreSingular} \cite[Theorem 5.8]{mont}
All regular sub-Riemannian \linebreak geodesics are normal and, therefore,  all abnormal geodesics are singular.
\end{prop}

If the horizontal distribution is fat, which means that for all $X \in \Ho$
$$\Ho + [X, \Ho ] = \Lie \, ,$$
then all sub-Riemannian geodesics are normal \cite[Proposition 4]{mont2}. For example, the horizontal distribution is fat in the case of $\SU (2)$, but not in the case of  $\SU (3)$.

Regarding the form of the normal geodesics we have the following result, which is \cite[Theorem 11.8]{mont} adapted to our setting. See also \cite{bosc1} and the references therein.

\begin{prop}\label{prop:GeodesicsForm} 
Consider a semi-simple, connected, compact Lie group $\G$ endowed with horizontal distribution defined by the orthogonal complement $\Ho$ of a Cartan subalgebra $\T$, and inner product \eqref{eq:InnerProduct}. 
Then the normal sub-Riemannian geodesics through the identity are of the form
\begin{equation}\label{eq:SubRiemannianGeodesics}
\gamma (t) = e^{tX} \cdot e^{- tX^{\perp}} \, ,
\end{equation}
where $X$ is any element of  $\Lie$ and $X^{\perp}$ is the orthogonal projection of $X$ onto $\T$.
\end{prop}

\begin{definition}\label{defn:HorizR}
If $X\in \Ho$, then we call $\gamma_X (t) = e^{tX}$ a horizontal Riemannian geodesic.  
\end{definition}

These are precisely the Riemannian geodesics which are also sub-Riemannian.
As we will see, they can be regular or singular.

If $R \in {\mathcal R}^+$, then let us use the notation
$\Ho_R = (\Lie_{R} \oplus \Lie_{-R} ) \cap \Lie$. With this notation we can rewrite \eqref{eq:LieAlgebraDecomposition} as
\begin{equation}\label{eq:TPerp}
\Ho = \T^{\perp} =  \bigoplus_{R \in \proots} \, \Ho_R \, .
 \end{equation}
From the relations
$$[\Ho_R , \Ho_R ] = \spa \{ R \}  \; \; \text{and} \; \; [R , \Ho_R ] = \Ho_R \, ,$$
we conclude that
\begin{equation}\label{eq:su(2)_R}
\su (2)_R = \Ho_R \bigoplus \spa \{ R \}
\end{equation}
is a subalgebra of $\Lie$, isomorphic to $\su (2)$.

For each $T \in \T$ let
$${\mathcal R}^T = \{ R \in {\mathcal R}^+ \, : \; \langle R , T \rangle =0 \} \, ,$$
and
\begin{equation}\label{eq:Lie^T}
\Lie^T = \bigoplus_{R \in {\mathcal R}^T} \, \su(2)_R \, .
\end{equation}
If ${\mathcal R}^T \neq \emptyset$, then $\Lie^T$ is a nontrivial Lie subalgebra of $\Lie$ and therefore we can find a closed, connected subgroup $\G^T$ of $\G$, which has $\Lie^T$ as its Lie algebra. Note that $\G^T$ carries a sub-Riemannian geometry, for which the horizontal distribution is 
\begin{equation}\label{eq:Ho^T}
\Ho^T = \bigoplus_{R \in {\mathcal R}^T} \, \Ho_R \, .
\end{equation}
Therefore, horizontal curves in $\G^T$ are also horizontal in $\G$ and if a normal sub-Riemannian geodesic of $\G$ lies in $\G^T$, then it is a normal sub-Riemannian geodesic of $\G^T$ too. 

A characteristic subgroup for a singular sub-Riemannian geodesic $\gamma$ is a closed connected subgroup within which $\gamma$ is regular.

\begin{prop}\label{prop:GeodesicsInG^T} \cite[Theorem 2]{mont1}
Every singular sub-Riemannian geodesic of $\G$ lies in some characteristic subgroup $\G^T$ with dimension strictly less than the dimension of $\G$.
\end{prop}

Propositions \ref{prop:AbnormalAreSingular} and \ref{prop:GeodesicsInG^T} allow us to give a simple algebraic proof of the following result, which is   also proved using control theoretic methods, including generalized Maslov index theory, in \cite{bosc1}.

\begin{theorem}\label{thm:SubRiemannianGeodesicsNormal}
Consider a semi-simple, connected, compact Lie group $\G$ endowed with the horizontal distribution defined by the orthogonal complement $\Ho$ of a Cartan subalgebra  $\T$ and inner product \eqref{eq:InnerProduct}.
Then we have the following results.\\
(i) All sub-Riemannian geodesics are normal.\\
(ii) All sub-Riemannian geodesics through the identity have the form
$$\gamma (t) = e^{tX} \cdot e^{- tX^{\perp}} \, ,$$
where $X \in \Lie$ and $X^{\perp}$ is the orthogonal projection of $X$ onto $\T$.
\end{theorem}

\begin{proof}
Let us assume that $\gamma$ is an abnormal sub-Riemannian geodesic of $\G$. Then, by Proposition \ref{prop:AbnormalAreSingular}, $\gamma$ is singular and by Proposition \ref{prop:GeodesicsInG^T}, there exists $T \in \T$ such that $\gamma$ lies in a characteristic subgroup $\G^T$. But, as $\gamma$ is regular in $\G^T$, 
by  Proposition \ref{prop:AbnormalAreSingular} it is also normal in $\G^T$. Hence,
 $\gamma$ must have the form \eqref{eq:SubRiemannianGeodesics} in $\G^T$, which, by \eqref{eq:TPerp}--\eqref{eq:Ho^T}, gives a normal sub-Riemannian geodesic of $\G$.

Once all sub-Riemannian geodesics are normal, part (ii) is a direct consequence of Proposition \ref{prop:GeodesicsForm}. 
\end{proof}

\section{Lengths of sub-Riemannian geodesic loops}\label{sec:3LengthsSR}
\noindent In this section we assume that $\G$ is a simple, simply connected, compact matrix Lie group.

For each root $R \in \roots$ and $n \in {\mathbb Z}$ we define the hyperplane in $\T$:
\begin{equation}\notag \label{eq:Hyperplane}
P(R , 2\pi n) = \{ T \in \T \, : \; \langle R , T \rangle = 2\pi n \} \, .
\end{equation} 
The reflections in $\T$ across  the hyperplanes $P(R, 0)$ will be denoted by $r_{R}$.
Note that 
$$r_{R} (T) = T - \frac{2 \langle R , T \rangle}{\lVert R\rVert^2} \, R \, .$$
The Weyl group of $\G$ can be defined as the group $W$ generated by the reflections $\{ r_{R} \, : \; R \in \roots \} \, .$ 

The set
$$\T \setminus \bigcup_{R \in \roots} P(R , 0) $$ is a union of disjoint, open cones, called Weyl  chambers. The Weyl group acts transitively on the Weyl chambers. We define the positive Weyl chamber by
\begin{equation}\notag
C = \{ T \in \T \, : \; \langle R , T \rangle > 0 \, , \; \forall \;   R \in \proots \} \, ,
\end{equation}
and let $\overline{C}$ denote its closure.

Let us choose the simple roots $\roots_s = \{ R_1, \dots , R_m \}$. 
In the case of a simple Lie algebra, the root system is irreducible and the length of the roots can take at most 2 values, which implies that the entries of the Cartan matrix,
$$N(R_j, R_k) = \frac{2 \langle R_j , R_k \rangle}{\lVert R_k \rVert^2} \, ,$$
can take only the following values:
$$N(R_j , R_k ) = \left\{ \begin{array} {l} 2 \; \text{if} \; \; j = k\\
0, -1 \; \text{and at most one of} \, -2 \, \text{or} \, -3  \, \text{if} \; j \neq k \, .\end{array}  \right.$$ 
For each $R \in \roots$ we denote by 
\begin{equation}\label{eq:P_R}
P_R = \frac{2\pi R}{\lVert R\rVert^2}
\end{equation}
the orthogonal projection of the origin onto the hyperplane $P(R , 2\pi)$.
It is known that \cite[Chapter 7, Lemma 7.6]{hel}
\begin{equation}\label{eq:e^2P_R=I}
e^{2P_R} = I \, , \: \forall \; R \in \roots \, .
\end{equation}

The unit lattice in $\T$ is defined by
\begin{equation}\notag \label{eq:UnitLattice}
\La_{\T} = \{ T \in \T \, : \; e^T = I \} \, ,
\end{equation}
and let us also set
\begin{equation}\notag \label{eq:Z_T}
{\mathcal Z}_{\T} = \left\{ n_1 2P_{R_1} +  \cdots + n_m 2P_{R_m} \, : \; n_1, \dots , n_m \in {\mathbb Z} \right\}\, .
\end{equation}
By the commutativity of $\T$, it is evident that ${\mathcal Z}_{\T} \subset \La_{\T}$.
By \cite[Theorem IX.1.6]{simon} we know that $\La_{\T} /\ {\mathcal Z}_{\T} \cong \pi_1 (\G)$, the fundamental group of $\G$.  Since $\G$ is simply connected, it follows that
\begin{equation}\label{L=Z}
\La_{\T} = {\mathcal Z}_{\T} \, .
\end{equation}
It is also known that \cite[Theorem IX.1.4]{simon}
\begin{equation}\label{eq:L_TSubset}
\La_{\T} \subseteq \{ T \in \T \, : \; \langle R , T \rangle \in 2\pi {\mathbb Z} \, , \; \forall \; R \in \roots \} \, ,
\end{equation}
and the two sets in \eqref{eq:L_TSubset} are equal only if the center of $\G$ equals $\{ I \}$.

\begin{definition} We call the numbers $n_1, \dots ,  n_m \in {\mathbb N} \cup \{0\}$ 
relatively prime if at least one of the numbers is non-zero and the greatest common factor  of the non-zero numbers is $1$. In particular, if we have only one non-zero number, then it must be $1$.
\end{definition}

\begin{remark}\label{prime}
By \eqref{L=Z}, if the numbers $n_1, \dots ,  n_m \in {\mathbb N} \cup \{0\}$ 
are relatively prime, then the line segment joining the origin to $n_1 2P_{R_1} +  \cdots + n_m 2P_{R_m} $ intersects $\La_{\T}$ only at the endpoints.
\end{remark}

\begin{theorem}\label{thm:RGeodesics/RPrime} Let $\G$ be a simple, simply connected, compact Lie group endowed with the bi-invariant inner product \eqref{eq:InnerProduct}.\\
(a) If the numbers $n_1, \dots ,  n_m \in {\mathbb N} \cup \{0\}$ are relatively prime and $T = n_1 2P_{R_1} + \cdots + n_m 2P_{R_m} $, then $\gamma_T (t)= e^{tT}$, $0 \leq t \leq 1$,  is a Riemannian geodesic loop with length $$\lVert n_1 2P_{R_1} +  \cdots + n_m 2P_{R_m} \rVert \, .$$
(b) All Riemannian geodesic loops in $\G$ have lengths
$$\lVert n_1 2P_{R_1} +  \cdots + n_m 2P_{R_m} \rVert \, ,$$ where $n_1, \dots ,  n_m \in {\mathbb N} \cup \{0\}$ are relatively prime.
\end{theorem}

\begin{proof} (a) If the rank of $\G$ is one, then $\To = \U (1)$ and any geodesic loop in $\To$ has length $2\pi$.
Now suppose the rank of $\G$ is greater than or equal to two.
Let $T = n_1 2P_{R_1} + \cdots + n_m 2P_{R_m} $, where $n_1 , \dots , n_m \in {\mathbb N} \cup \{0\}$ are relatively prime and $\gamma_T (t) = e^{tT}$. By the commutativity of the elements of $\T$ we know that $\gamma_T (1) = I$. If, for some $0<t<1$, we have $\gamma_T (t)=I$, then
$ t (n_1 2P_{R_1} + \cdots + n_m 2P_{R_m} ) \in \La_{\T}$ and, by Remark \ref{prime}, this contradicts the fact that 
$n_1 , \dots , n_m$ are relatively prime.
Hence, the length of one loop described by $\gamma_T$ is 
\begin{equation}\label{eq:Length(gamma_T)1}
\Le (\gamma_T ) = \int_0^1 \, \lVert T \Vert \, dt = \lVert n_1 2P_{R_1} + \cdots + n_m 2P_{R_m} \rVert \, .
\end{equation}
(b) Let $X \in \Lie$ and $\gamma_X (t) = e^{tX}$. Assume that
$\gamma_X (1) = I$ and $\gamma_X (t) \neq I$ if $0<t<1$. By the facts that $\Ad (\G )(X) \cap \T$ is non-empty and finite, the Weyl group acts transitively on the Weyl chambers, and each element of the Weyl group can be written as $\Ad (g)$ for some $g \in \G$, it follows that there exists $g \in \G$ such that $T_X = \Ad (g)X \in \overline{C}$. 
Hence, $e^{T_X} =ge^Xg^{-1} = I$ and therefore $T_X \in \La_{\T}$.  By \eqref{L=Z} we obtain that
$T_X = n_1 2P_{R_1} + \cdots + n_m 2P_{R_m} $, where $n_1 , \dots , n_m \in {\mathbb N} \cup \{0\}$ are relatively prime. Using the fact that $\lVert T_X \rVert = \lVert X \rVert$ we find that 
\begin{equation}\notag \label{eq:Length(gamma_X)}
\Le (\gamma_X ) = \int_0^1 \, \lVert T_X \rVert \, dt =\lVert n_1 2P_{R_1} + \cdots + n_m 2P_{R_m} \rVert \, .
\end{equation}

\end{proof}

\begin{remark}
Moreover, for any $0\neq T = n_1 2P_{R_1} + \cdots + n_m 2P_{R_m} $ we have that  $\Ad (\G )(T) \cap (\Lie \setminus \T) \neq \emptyset $, so there exists $X \not\in \T$ in the same conjugacy class with $T$. Hence we have a  Riemannian geodesic loop outside of $\To$, corresponding to $X$, which has length equal to $\lVert T \rVert$ in \eqref{eq:Length(gamma_T)1}.
\end{remark}

We need the following lemma to generalize Theorem \ref{thm:RGeodesics/RPrime} to the case of horizontal Riemannian geodesic loops (see Definition \ref{defn:HorizR}). 

\begin{lemma}
For any $T \in \T$ we have $\Ad (\G )(T) \cap \Ho \neq \emptyset $.
\end{lemma}

\begin{proof}
By \cite[Lemma 2.2]{dand}, given $\T$,  we can construct another Cartan subalgebra $\T'$ which is orthogonal to $\T$. Hence, $\T' \subset \Ho$ and, as any two Cartan subalgebras are conjugate,  there exists some $g \in \G$ such that $\Ad (g) \T = \T'$. Hence, we conclude that for any $T \in \La_{\T}$ we have that  $\Ad (\G )(T) \cap \Ho \neq \emptyset $.
\end{proof}

\begin{theorem}\label{thm:PurelySubRLengths}  
Consider a semi-simple, connected, compact Lie group $\G$ endowed with horizontal distribution defined by the orthogonal complement $\Ho$ of a Cartan subalgebra $\T$, and inner product \eqref{eq:InnerProduct}. Then the horizontal Riemannian geodesic loops have lengths
$$\lVert n_1 2P_{R_1} +  \cdots + n_m 2P_{R_m} \rVert \, ,$$ where $n_1, \dots ,  n_m\in {\mathbb N} \cup \{0\}$ are relatively prime.
\end{theorem}

\begin{proof} Let $X \in \Ho$ and $\gamma _X (t) = e^{tX}$. If $\gamma_X (1) = I$ and
$\gamma_X (t) \neq I$ for all $0<t<1$, then we can follow the proof of Theorem \ref{thm:RGeodesics/RPrime} (b), to conclude that there exist $n_1 , \dots , n_m \in {\mathbb N} \cup \{0\}$ relatively prime  such that
$$\Le (\gamma_X ) = \lVert n_1 P_{R_1} + \cdots + n_m P_{R_m} \rVert \, .$$
By Lemma \ref{cor:RootMax}, the entire length spectrum of $\lVert n_1 P_{R_1} + \cdots + n_m P_{R_m} \rVert$, where $n_1 , \dots , n_m \in {\mathbb N} \cup \{0\}$ are relatively prime, is covered, and this finishes the proof.
\end{proof}

One might expect the shortest sub-Riemannian geodesic loops to be longer than their Riemannian counterparts.  Surprisingly, the following result,  which generalizes the Riemannian case of \cite[Chapter 7, Proposition 11.9]{hel}, proves otherwise.   

\begin{theorem}\label{thm:shortestlengths}
The shortest sub-Riemannian geodesic loops are also the shortest Riemannian geodesic loops.  Their common length is
 $\frac{4\pi}{\lVert {R^*} \rVert}$, where ${R^*}$ is the highest root. 
\end{theorem}

\begin{proof} We first consider the Riemannian case.  Without loss of generality we can assume that the rank of $\Lie$ is greater than $1$. 
Let $\gamma^* (t) = e^{t \, 2P_{R^*} } \, .$
By \eqref{eq:e^2P_R=I} we know that $\gamma^* (1) = I$.  Moreover,  there exists $R \in \proots$ such that  
$$N(R, R^* ) = 2 \frac{\langle R, R^* \rangle}{\lVert R^* \rVert^2} = 1 \, .$$ Therefore, for any $0 < t <1$ we have
$$\langle {R} , t \, 2 P_{R^*}  \rangle = 2\pi t\, ,$$
which, by \eqref{eq:L_TSubset}, implies that $\gamma^* (t) \neq I$ if $ 0<t<1$.
Hence, the length of one loop described by $\gamma^*$ is 
\begin{equation}\notag \label{eq:Length(gamma^*)}
\Le (\gamma^* ) = \int_0^1 \, \lVert 2 P_{R^*} \rVert \, dt = \frac{4\pi}{\lVert {R^*} \rVert} \, .
\end{equation}
Let $T = n_1 2P_{R_1} + \cdots + n_m 2P_{R_m} $, where $n_1 , \dots , n_m \in {\mathbb N} \cup \{0\}$ are relatively prime and let $\gamma_T (t) = e^{tT}$.
Assume that $\gamma_T (t) \neq I$ if $0<t<1$ and 
$$\Le (\gamma_T ) \leq \Le (\gamma^* ) \, .$$
Hence, 
$$\lVert T \rVert \leq \frac{4\pi}{\lVert R^* \rVert} = \lVert 2 P_{R^*} \rVert \, .$$
As in the proof of Theorem \ref{thm:RGeodesics/RPrime}, by the fact that the Weyl group acts transitively on the Weyl chambers, there exist $g \in \G$ and $T_1 \in \overline{C}$ such that $T_1 = \Ad (g) T$.
Therefore, $e^{T_1} = I$ and hence $\langle R^* , T_1 \rangle = 2\pi n$ for some $n \in {\mathbb N}$. By \cite[Chapter 7, Theorem 6.1]{hel},
$$P(R^* , 2\pi) \cap \overline{C} \cap \La_{\T} = \emptyset \, ,$$
which implies that $n \neq 1$. On the other hand, $\lVert T_1 \rVert = \lVert T \rVert \leq \lVert 2 P_{R^*} \rVert$,
which is the shortest distance from the origin to $P(R^*, 4\pi)$. Therefore, $n=2$ and this implies that
$T_1 = 2 P_{R^*}$. In conclusion, we have $\Le (\gamma_T) = \frac{4\pi}{\lVert R^* \rVert}$,  which establishes the length of the shortest Riemannian geodesic loops. Note that this slight generalization of \cite[Chapter 7, Proposition 11.9]{hel} is proved differently here.

We now consider the sub-Riemannian case.  Theorem \ref{thm:PurelySubRLengths} implies that the shortest horizontal Riemannian geodesic loops have length equal to $\frac{4\pi}{\lVert {R^*}\rVert}$, which equals the length of the shortest Riemannian geodesic loops by the argument above. By Remark \ref{remark:Smooth} every sub-Riemannian geodesic is a smooth Riemannian curve of equal length, so we conclude that
$\frac{4\pi}{\lVert {R^*}\rVert}$ is the shortest length for any sub-Riemannian geodesic loop.
\end{proof}

Theorem \ref{thm:shortestlengths} implies the following corollary regarding the length of the highest root.

\begin{cor}\label{cor:RootMax}
We have
$$\lVert R^* \rVert = \max \frac{4\pi}{\lVert n_1 2 P_{R_1} + \cdots + n_m P_{R_m} \rVert }\, ,$$
where   $n_1, \dots ,  n_m \in {\mathbb N} \cup \{0\}$ are relatively prime.
\end{cor}

Regarding the sub-Riemannian geodesic loops which are not necessarily horizontal Riemannian, we have the following result. 

\begin{theorem}\label{thm:Xdecomposition}
Let $X = H + X^{\perp}$ such that $H \in \Ho$ and $X^{\perp} \in \T$. 
Consider $\gamma (t) = e^{tX} \cdot e^{-tX^{\perp}}$ and 
assume that  $\gamma (t) \neq I$ if $0<t<1$ and  $\gamma (1) = I$.
Then:\\
(a) The length of $\gamma $ satisfies
$$\Le (\gamma ) = \lVert H \rVert \geq \frac{4\pi}{\lVert R^* \rVert} \, ,$$
and there is an $X = H + X^{\perp}$ for which $\frac{4\pi}{\lVert R^* \rVert}$ is attained.\\ 
(b) We have
\begin{equation}\label{eq:H}
H = e^{X^{\perp} } \, H \, e^{-X^{\perp}} \, .
\end{equation} 
(c) If $$\Ad (\G ) (X^{\perp} ) \cap \T = \{ S_1, \dots , S_l \} ,$$ 
then for all $1 \leq j \leq l$ there exist $L_j \in \La_{\T}$ such that 
$$\Ad (\G ) (X ) \cap \T = \{ S_1+L_1, \dots , S_l+L_l \} \, .$$ 
\end{theorem} 

\begin{proof}
(a) Note that  $\gamma (1) = I$ implies that $e^X = e^{X^{\perp}}$. Then,
\begin{equation}\label{eq:GammaPrime}
\begin{split}
\gamma'(t)  = & e^{tX} \cdot X \cdot e^{-tX^{\perp}} - e^{tX} \cdot X^{\perp} \cdot e^{-tX^{\perp}} \\ = & e^{tX} (X -X^{\perp} ) \, e^{-tX^{\perp}}  ,
\end{split}
\end{equation}
and
\begin{equation*}
\begin{split}
\lVert \gamma'(t) \rVert_{\gamma (t)} = & \lVert \gamma(t)^{-1} \cdot \gamma'(t)\rVert_{I} \\ = & \lVert e^{tX^{\perp}} \cdot (X - X^{\perp} )
\cdot e^{-tX^{\perp}}\rVert = \lVert X - X^{\perp} \rVert \, .
\end{split}
\end{equation*}
Hence,  the length of $\gamma$ is $\lVert X-X^{\perp} \rVert = \lVert H \rVert$.
The fact that $\lVert H \rVert \geq \frac{4\pi}{\lVert R^* \rVert}$ is an immediate consequence of 
Proposition \ref{thm:PurelySubRLengths}. 

Consider the simply connected Lie subgroup $\SU (2)_{R^*}$ of $\G$ which has its Lie algebra equal to $\su (2)_{R^*}$, and denote by $X^*$ and $Y^*$ those elements of \eqref{eq:HorizontalBasis} which, together with $R^*$, generate $\su (2)_{R^*}$.
The relations 
$$[R^*,  X^* ] = - \lVert R^* \rVert^2 Y^*   \; \text{and} \; [R^*,  Y^* ] =  \lVert R^* \rVert^2 X^* $$
show that the only positive root, and hence the highest root in $\su (2)_{R^*}$, is $R^*$. In a similar way to the proof of \eqref{eq:HLength}, we can obtain a sub-Riemannian geodesic loop in $\SU (2)_{R^*}$ whose length is $\frac{4\pi}{\lVert R^* \rVert}$.

(b) We claim that $\gamma' (0) = \gamma' (1)$. This information can be found in \cite[Page 148, Exercise 3]{hel} and its proof is based on the fact that for all $t \in {\mathbb R}$,
\begin{equation*}
\begin{split}
\gamma (t+1) = & \; e^{(t+1)X} \, e^{-(t+1)X^{\perp}} = e^{tX} \, e^{X} \, e^{-tX^{\perp}} \, e^{-X^{\perp}}\\ = & \; e^{tX} \, e^{X} \, e^{-X^{\perp}} \, e^{-tX^{\perp}} = e^{tX} \, e^{-tX^{\perp}} = \gamma(t) \, .
\end{split}
\end{equation*}
By \eqref{eq:GammaPrime}, it follows that
$$ X-X^{\perp} = e^{X} \, (X-X^{\perp} ) \, e^{-X^{\perp}} \, ,$$
which clearly implies \eqref{eq:H}.

(c) By the properties of the adjoint representation there exist \linebreak $S_1, \dots , S_l \in \T$ and $S'_1 , \dots , S'_l \in \T$, where $l$ is the number of Weyl chambers, such that
\begin{equation}\label{eq:Ad(G)1}
\Ad (\G ) (X^{\perp}) \cap \T = \{ S_1, \dots , S_l \} \, 
\end{equation}
and
\begin{equation}\label{eq:Ad(G)2}
\Ad (\G ) (X) \cap \T = \{ S'_1, \dots , S'_l \} \, .
\end{equation}
Note that in \eqref{eq:Ad(G)1} and \eqref{eq:Ad(G)2} some of the $S_j$ and $S'_j$ might be repeated if they belong to one of the hyperplanes $P(R,0)$ for $R \in \roots$.

Therefore,
\begin{equation}\label{eq:Ad(G)3}
\Ad (\G ) (e^{X^{\perp}}) \cap \To = \{ e^{S_1}, \dots , e^{S_l} \} \, ,
\end{equation}
and
\begin{equation}\label{eq:Ad(G)4}
\Ad (\G ) (e^X) \cap \To = \{ e^{S'_1}, \dots , e^{S'_l} \} \, .
\end{equation}
The fact that $e^{X^{\perp}} = e^X$ implies that the sets in \eqref{eq:Ad(G)3} and \eqref{eq:Ad(G)4} must coincide. Hence, by rearranging the elements if necessary, we can suppose that for all $1 \leq j \leq l$ we have $e^{S_j} = e^{S'_j}$, which immediately implies the existence of $L_j \in \La_{\T}$ such that $S'_j = S_j + L_j $. 
\end{proof}

Since $X^{\perp} \in \T$, we can see that one of $S_1, \dots , S_l$ in \eqref{eq:Ad(G)1} must be $X^{\perp}$. Therefore, we have the following corollary.

\begin{cor}\label{cor:Xdecomp2}
Under the assumptions of Theorem \ref{thm:Xdecomposition}, there exist $g \in \G$ and $L \in \La_{\T}$ such that 
$X = \Ad (g) (X^{\perp} + L) \, .$
\end{cor}

\section{The case of $\SU (2)$}\label{sec:4SU(2)}
\noindent The special unitary group of $2 \times 2$ complex matrices is
$$\SU (2) = \left\{ g = \left( \begin{array}{cc} \alpha & \beta\\ - \overline{\beta} & \overline{\alpha} \end{array}\right)  \; : \; \alpha, \beta \in \comp \, , \; \lvert \alpha \rvert^2+ \lvert \beta \rvert^2 =1 \right\} \, .$$
Its Lie algebra is the three dimensional real Lie algebra
$$\su (2) = \left\{ X = \left( \begin{array}{cc} i x_1 & x_2+i x_3\\-x_2 + i x_3 & -ix_1 \end{array} \right) \; : \; x_1 , x_2 , x_3 \in \real \right\} \, .$$
The Killing form of $\su (2)$ is
$$K(X,Y) = 4 \tr (XY) \, ,$$
while the inner product \eqref{eq:InnerProduct} is defined as
$$\langle X , Y \rangle = - \frac{1}{2} \tr (XY) \, .$$
The Cartan subalgebra $\T$ is spanned by the unit vector
$$T_1 = \begin{pmatrix} - i & 0 \\ 0 & i \end{pmatrix}  \, $$
and the orthonormal basis of $\Ho$ is formed by
$$X_1 = \begin{pmatrix}0 & 1\\-1& 0 \end{pmatrix}  \,  \; \text{and} \;  Y_1= \begin{pmatrix} 0 & i\\i& 0 \end{pmatrix} \, .$$ 
The exponential map $\exp \colon \su (2) \to \SU (2)$ has the following simple form:
\begin{equation}\notag \label{eq:SU(2)e^X}
 \exp (X) = e^X = \cos (\lVert X \rVert) \, I + \frac{\sin (\lVert X \rVert)}{\lVert X\rVert} \, X \, .
\end{equation}
Consider $X = aX_1 + bY_1 +cT_1$. Then, for $0 \leq t \leq 1$,
$$e^{tX} = \cos (t\sqrt{a^2+b^2+c^2}) \, I + \frac{\sin (t\sqrt{a^2+b^2+c^2})}{\sqrt{a^2+b^2+c^2}} \, X ,$$
and
$$e^{tX^{\perp}} = \cos (tc) \, I + \sin (tc) \, T_1 .$$
The Riemannian geodesic $e^{tX}$ closes the first time at $t=1$ if  we have $\sqrt{a^2+b^2+c^2} = 2\pi$, which shows that the Riemannian length spectrum equals $\{2\pi\}$.
For the sub-Riemannian geodesic $\gamma (t) = e^{tX} e^{-tX^{\perp}}$, the condition $e^X = e^{X^{\perp}}$ implies that
\begin{equation}\label{eq:SU(2)sqrtCondition}
\sqrt{a^2+b^2+c^2} = n \pi  \; \text{and} \; \lvert c \rvert = m \pi  \, ,
\end{equation}
where $m,n \in {\mathbb N} \cup \{0\} $, $m \leq n$, are both even or both odd.
In order that $\gamma (t) \neq I$ for all $0<t<1$, we have to require that
$m,  n \in {\mathbb N} \cup \{0\}$ are both odd and relatively prime or both even and $\frac{m}{2},  \frac{n}{2}$ are relatively prime with one of them odd and the other even.

Notice that, as the only positive root of $\SU (2)$ is $R_1 = 2T_1$, we have $\lVert R^* \rVert = \lVert 2T_1 \rVert = 2$, and therefore
\begin{equation}\label{eq:HLength}
\lVert H \rVert = \lVert X - X^{\perp} \rVert = \pi \sqrt{n^2 - m^2} = \frac{2\pi \sqrt{n^2-m^2}}{\lVert R^* \rVert}\, .
\end{equation}

The same result can be obtained by  Theorem \ref{thm:Xdecomposition}. In $\SU (2)$ the unit lattice is
$\La_{\T} = \{ 2k \pi T_1 \; : \: k \in {\mathbb Z} \} \, .$
Formula \eqref{eq:H} implies that 
$$c = m \pi \in \pi {\mathbb Z}\, .$$
The matrices $S'_1$ and $S'_2$ from \eqref{eq:Ad(G)2} are diagonal with entries consisting of the eigenvalues of $X$. Thus, 
Theorem \ref{thm:Xdecomposition} implies that there exists some $k \in {\mathbb N}$ such that
\begin{equation}\notag \label{eq:SU(2)sqrtCondition2}
\sqrt{a^2+b^2+m^2 \pi^2} - m\pi = 2k \pi \, ,
\end{equation}
and this implies \eqref{eq:SU(2)sqrtCondition}.

We have therefore presented two algebraic proofs of the following proposition, which is a special case of Theorems \ref{thm:RGeodesics/RPrime}, \ref{thm:PurelySubRLengths}, and \ref{thm:shortestlengths}, and which extends the results from \cite{chang,vanval}.

\begin{prop}\label{prop41} In $\SU (2)$ the following properties hold.\\
(a) The Riemannian geodesic loops have length equal to $2\pi$.\\ 
(b) The horizontal Riemannian geodesic loops have length equal to $2\pi$.\\
(c) The shortest sub-Riemannian geodesic loops have length equal to $2\pi$.\\
(d) The sub-Riemannian geodesic loops have lengths equal to $\pi \sqrt{n^2 - m^2}$, where
$m,  n \in {\mathbb N} \cup \{0\}$ are odd and relatively prime or even and $\frac{m}{2},  \frac{n}{2}$ are relatively prime with one of them odd and the other even.

\end{prop}

\begin{remark}
As an introduction to the next section, let us show that we can use Vi\`{e}te's formulas to get the result of Proposition \ref{prop41} (d). Indeed, the characteristic polynomial of $X$ is $P (\lambda) = \lambda^2 + (a^2 +b^2+c^2)$, and by Theorem 3.3 and the first Vi\`{e}te's formula, the eigenvalues of $X$ must be of the form $\lambda_1 = -ci - 2k\pi i$ and $\lambda_2 = ci + 2k\pi i$, where $ k \in {\mathbb N}$.
The second Vi\`{e}te's formula gives
$$\lambda_1 \lambda_2 = (c+2k\pi)^2 = a^2 +b^2+c^2 \, ,$$
which leads to \eqref{eq:SU(2)sqrtCondition}.
\end{remark}

\begin{remark}
For comparison with the case of $\SU (3)$ in the next section, note that in $\SU (2)$ the sub-Riemannian geodesic loops have the form
\begin{equation}\label{eq:SU(2)geodesicroot}
\gamma (t) = e^{t(a_1 X_1 + b_1 Y_1 + \frac{m}{2} R_1)} \, e^{-t\frac{m}{2} R_1},
\end{equation} 
where $a, b, c, m$ satisfy \eqref{eq:SU(2)sqrtCondition}.
\end{remark}

\section{The case of $\SU (3)$}\label{sec:5SU(3)}
\noindent Consider the special unitary group of $3 \times 3$ complex matrices
$$\SU (3) = \{ g \in \GL (3, \comp) \, : \; g \cdot g^* = I \, , \; \det g = 1 \} \, ,$$
and its Lie algebra
$$\su (3) = \{ X \in \gl (3, \comp) \, : \; X + X^* = 0 \, , \; \tr X =0 \} \, .$$
The inner product is defined by 
$$\langle X , Y \rangle = - \frac{1}{2} \tr ( X  Y ) \, .$$
We consider the maximal torus
$$\To = \left\{  \begin{pmatrix}  e^{ia_1} & 0 & 0 \\ 0 & e^{ia_2} & 0 \\ 0 & 0 & e^{ia_3} \end{pmatrix} \; : \; a_1, \, a_2, \, a_3 \in \real \, , \; a_1 + a_2 + a_3 = 0 \right\}$$
and its Lie algebra 
$$\T = \left\{  \begin{pmatrix} ia_1 & 0 & 0 \\ 0& ia_2 & 0 \\ 0&0& ia_3 \end{pmatrix} \; : \; a_1, \, a_2, \, a_3 \in \real \,  , \; a_1 + a_2 + a_3 = 0 \right\},$$
 which is our choice for the Cartan subalgebra.
The following are the Gell-Mann matrices, which form an
orthonormal basis of $\su (3)$ and satisfy the relations in formulas \eqref{eq:LieAlgebraDecomposition}--\eqref{eq:HorizontalBasisProperties}:

\begin{align*}
&T_1 = \begin{pmatrix} -i  & 0 & 0 \\ 0& i  & 0 \\ 0&0& 0 \end{pmatrix}, 
&&T_2 = \begin{pmatrix} \frac{-i}{\sqrt{3}}  & 0 & 0 \\ 0& \frac{-i}{\sqrt{3}}  & 0 \\ 0&0& \frac{2i}{\sqrt{3}} \end{pmatrix}, \\[0.2cm]
&X_1 = \begin{pmatrix} 0  & 1 & 0 \\ -1 & 0  & 0 \\ 0&0& 0 \end{pmatrix}, 
&&Y_1 = \begin{pmatrix} 0 & i & 0 \\  i & 0 & 0 \\ 0&0& 0 \end{pmatrix},
\\[0.2cm]
&X_2 = \begin{pmatrix} 0  & 0 & 0 \\ 0 & 0  & 1 \\ 0&-1& 0 \end{pmatrix}, 
&&Y_2 = \begin{pmatrix} 0 & 0 & 0 \\  0 &  0 & -i \\ 0&-i& 0 \end{pmatrix},
\\[0.2cm]
&X_3 = \begin{pmatrix} 0  & 0 & 1 \\ 0 & 0  & 0 \\ -1&0& 0 \end{pmatrix},  
&&Y_3 = \begin{pmatrix} 0 & 0 & i \\  0 &  0 & 0 \\ i&0& 0 \end{pmatrix}.
\end{align*}

The positive roots are the following:
\begin{align*}
& R_1 = \begin{pmatrix} - 2i  & 0 & 0 \\ 0 & 2i  & 0 \\ 0&0& 0 \end{pmatrix} = 2T_1 , \\[0.2cm] 
& R_2 = \begin{pmatrix} 0 & 0 & 0 \\  0 &  2i & 0 \\  0 & 0 & -2i \end{pmatrix}  = T_1 - \sqrt{3} T_2  ,\\[0.2cm]
& R_3 = \begin{pmatrix}-2i & 0 & 0 \\  0 &  0 & 0 \\  0 & 0 & 2i \end{pmatrix}  = T_1 + \sqrt{3} T_2  .
\end{align*}
The highest root is ${R^*} = R_1$, while the two simple roots are $R_2$ and 
$R_3$.
The unit lattice is
$$\La_{\T} = \{ n \pi R_2 + m \pi R_3 \, : \; n \, , \; m \in {\mathbb Z} \} \, ,$$
and observe that 
\begin{equation}\label{eq:SU(3)exp}
e^{\pi R_1} = e^{\pi R_2} = e^{\pi R_3} = I \, .
\end{equation}
For $k = 1,2,3$, the projections of the origin onto the hyperplanes $P(R_k , 2\pi )$ are 
$$P_{R_k} = \frac{\pi}{2} R_k \, ,$$
and, indeed, \eqref{eq:SU(3)exp} is equivalent to
$$e^{2P_{R_k }}= I \, , \; k=1,2,3.$$

Observing that 
$$\lVert n \pi R_2 + m \pi R_3 \rVert = 2\pi \sqrt{n^2 - nm +m^2} \, ,$$
we conclude that, in $\SU (3)$, Theorems \ref{thm:RGeodesics/RPrime}, \ref{thm:PurelySubRLengths}, and \ref{thm:shortestlengths} have the following special form.

\begin{prop} In  $\SU (3)$ the following properties hold.\\
(a) The Riemannian geodesic loops have lengths equal to $$ 2\pi \sqrt{n^2 - nm +m^2} \, ,$$ where $n ,m \in {\mathbb N} \cup \{0\}$ are relatively prime.\\
(b) The horizontal Riemannian geodesic loops have lengths equal to $$ 2\pi \sqrt{n^2 - nm +m^2} \, ,$$
where $n ,m \in {\mathbb N} \cup \{0\}$ are relatively prime.\\
(c) The shortest sub-Riemannian geodesic loops have length equal to $2\pi$.
\end{prop}

To obtain information about the full sub-Riemannian length spectrum in $\SU (3)$, consider 
\begin{align*}
& H = a_1 X_1 + b_1 Y_1 + a_2 X_2 + b_2 Y_2 + a_3 X_3 + b_3 Y_3, \\
&X^{\perp} = \frac{c_1}{2} R_3 + \frac{c_2}{2} R_2 = \left( \begin{array}{ccc}  - c_1 i & 0 &0\\ 0 & c_2 i &0\\0 & 0 & (c_1 -c_2)i \end{array} \right) \, ,\\
& X = H + X^{\perp} , \; \text{and} \; \gamma (t) = e^{tX} e^{-tX^{\perp}} .
\end{align*}

The characteristic polynomial of $X$ is
$$P(\lambda) = -\lambda^3 - p \lambda + qi ,$$
where
$$p = \sum_{j=1}^3 (a_j^2 + b_j^2) + c_1^2 + c_2^2 - c_1 c_2  ,$$
and
\begin{align*}
q = \, & c_2 \left( a_3^2+b_3^2 - a_1^2 - b_1^2 + c_1^2 \right) \\
- & c_1 \left( a_2^2+b_2^2 - a_1^2 - b_1^2 + c_2^2 \right) \\
+ & 2 \left(a_1 a_2 b_3 + a_1 b_2 a_3 - b_1 a_2 a_3 + b_1 b_2 b_3 \right) .
\end{align*}
Note that
$p = \lVert H \rVert^2 + \lVert X^{\perp} \rVert^2 .$
Formula \eqref{eq:H} gives
\begin{align}\label{c1c2}
c_1 + c_2  &\in 2\pi {\mathbb Z} \; \; \text{if} \; \; a_1+b_1 i \neq 0 \, , \nonumber\\
c_1 -  2c_2 &\in 2\pi {\mathbb Z} \; \;  \text{if} \; \; a_2+b_2 i \neq 0 \, ,\\
2c_1  - c_2  &\in 2\pi {\mathbb Z} \; \; \text{if} \; \; a_3+b_3 i \neq 0 \, . \nonumber
\end{align}

To see the connection with the case of $\SU (2)$, let us start with the following simple cases.

{\bf Case 1.} Consider $c_2 = - c_1$, $a_2 =b_2 = a_3 = b_3 =0$. This corresponds to the case of $\SU (2)$ from the previous section and these geodesics are singular in $\SU (3)$.  Therefore the sub-Riemannian geodesics have the form \eqref{eq:SU(2)geodesicroot} and the lengths $\pi \sqrt{n^2-m^2}$ from 
Proposition \ref{prop41} (d).

{\bf Case 2.}  Consider $c_2 =0$, $a_2 =b_2 = a_3 = b_3 =0$. These geodesics are not contained in any copy of $\SU (2)$ and are regular in $\SU (3)$.
Here, $c_1 = 2m\pi$, $m \in {\mathbb Z} $, and the eigenvalues of $X $ are
$$- 2m\pi i \; \; \text{and} \; \; \left(m \pi \pm \sqrt{a_1^2+ b_1^2 + m^2 \pi^2}\right)  i \, .$$ 
By Theorem \ref{thm:Xdecomposition} (c) we have that
\begin{equation}\label{eq:|c_1|}
 \lvert c_1 \rvert = 2m\pi \; \text{and} \; a_1^2+b_1^2 + m^2 \pi^2 = n^2 \pi^2 \, ,
\end{equation}
where $m,n \in {\mathbb N} \cup \{0\} $, $m \leq n$ are both odd or even. Therefore, the sub-Riemannian geodesic loop corresponding to \eqref{eq:|c_1|} is
$$\gamma (t) = e^{t(a_1 X_1 + b_1 Y_1 + m R_3)} \, e^{-tm R_3} \, ,$$
and its length is $\pi \sqrt{n^2-m^2}$. 

{\bf Case 3.} If at least two of $a_j + b_j i$, $j=1,2,3$, are not zero, then
$$c_1 = \frac{4n+2m}{3} \pi , \; c_2 = \frac{2n-2m}{3} \pi \, , $$
where $n,m \in {\mathbb Z}$.
Theorem \ref{thm:Xdecomposition} (c) and the first Vi\`{e}te formula for the characteristic polynomial imply that the eigenvalues of $X$ must have the form
\begin{align}\label{lambda}
&\lambda_1 = \left( -c_1 - 2k\pi \right) i , \nonumber\\
&\lambda_2 = \left( c_2 + 2l \pi \right) i , \\
&\lambda_3 = \left( c_1-c_2 + 2 (k-l) \pi \right) i . \nonumber
\end{align}
The second Vi\`e{t}e formula gives
\begin{equation}\label{normH}
\lVert H \rVert  = \sqrt{ c_1 (4k-2l)\pi +  c_2 (4l-2k)\pi + 4 (k^2+l^2-kl ) \pi^2} .
\end{equation}
From the third Vi\`{e}te formula we find
\begin{multline}\label{viete3}
4 c_1c_2 (k-l) \pi + 4 c_1 l(2k-l) \pi^2 + 4 c_2 k (k-2l) \pi^2 \\  \shoveright{+ 2 c_1^2 l \pi - 2 c_2^2 k \pi  + 8kl(k-l)\pi^3} \\
\shoveleft = (a_3^2 + b_3^2 - a_1^2 - b_1^2 ) c_2 - (a_2^2+b_2^2 - a_1^2 - b_1^2 ) c_1\\ +
2(a_1 a_2 b_3 + a_1 b_2 a_3 - b_1 a_2 a_3 + b_1 b_2 b_3 ) \, .
\end{multline} 

The complexity of  formula \eqref{viete3} hides its true geometric meaning. However, in the case when $q=0$,  formula \eqref{viete3} reduces to $0=0$, and we have the following eigenvalues for $X$:
$$ 0 \; \; \text{and} \; \; \pm \sqrt{\lVert H \rVert^2 + \lVert X^{\perp} \rVert^2}\,  i .$$
Without loss of generality we can assume that $\lambda_1 = 0$. Then $c_1 = 2k\pi $, which implies that $m = 3k-2n$ and $c_2  = 2(n-k)$. 
From \eqref{normH} it follows that
\begin{equation*}
\lVert H \rVert  = 2\pi  \sqrt{ (2k-l)^2 -nk + 2nl },
\end{equation*}
which in the case of $k=l$ reduces to
\begin{equation*}
\lVert H \rVert  = 2\pi  \sqrt{ k^2 + nk} = 2\pi \sqrt{\left( k+\frac{n}{2} \right)^2 - \frac{n^2}{4} } .
\end{equation*}
This shows that, as expected, formula  \eqref{normH} includes the sub-Riemannian geodesic loop length spectrum of $\SU (2)$.

Note that $q=0$ is satisfied if $c_2 = 0$, $a_1^2 + b_1^2 = a_2^2 + b_2^2$ and $a_3 = b_3 =0$. As a numerical example we can give the sub-Riemannian geodesic loop of length $8\pi$ described by
$$\gamma (t) = e^{\pi (5X_1 + \sqrt{7}Y_1 + 5 X_2 +\sqrt{7}Y_2 + 3 R_3)t} e^{-3\pi R_3 t} \, .$$

In conclusion, we have the following result.
\begin{prop}
In $\SU(3)$ the sub-Riemannian geodesic loops have lengths equal to
\begin{equation*}
2\pi \sqrt{ \left( \frac{(2n+m) (2k-l)}{3} +  \frac{(n-m) (2l-k)}{3} \right) 
+  (k^2+l^2-kl ) } \, ,
\end{equation*}
where $m,n,k,l \in {\mathbb Z}$.
\end{prop}

\subsection*{Acknowledgements}
The authors are grateful to Richard Montgomery for helpful discussions and advice on exposition.  We also thank the referees for valuable suggestions.

\end{document}